\documentclass[10pt,reqno,a4paper]{amsart}
\usepackage{amsmath}
\usepackage{enumerate}
\makeatletter
\@namedef{subjclassname@2010}{%
  \textup{2010} Mathematics Subject Classification}
\makeatother
\newtheorem{theorem}{Theorem}[section]
\newtheorem{corollary}{Corollary}[section]
\newtheorem{lemma}{Lemma}[section]

\theoremstyle{definition}

\theoremstyle{remark}
\newtheorem{remark}{Remark}[section]
\numberwithin{equation}{section}

\title[The explicit evaluations of  Ramanujan's remarkable product of theta-functions $a_{m, n}$]
      {The explicit evaluations of  Ramanujan's remarkable product of theta-functions $a_{m, n}$}

\author[D. J. Prabhakaran]{D. J. Prabhakaran}
\address{Department of Mathematics \\
Anna University, MIT Campus\\
Chennai- 600025\\ India}
\email{asirprabha@gmail.com}
\author[K. Ranjith kumar]{K. Ranjith kumar}
\address{Department of Mathematics \\
Anna University, MIT Campus\\
Chennai- 600025\\ India}
\email{ranjithkrkkumar@gmail.com}
\begin{document}
\begin{abstract}
On pages 338 and 339 in his first notebook, Ramanujan defined the remarkable product of theta-functions $a_{m, n}$. Also he recorded eighteen explicit values depending on two parameters, namely, $m$, and $n$, where these are odd integers.  All these values were established by Berndt et al. In this paper, we initiate to study explicit evaluations of $a_{m, n}$ for even values of $m$. We establish a new general theorems for the explicit evaluations of $a_{m, n}$ involving class invariant of Ramanujan. For this purpose, we establish several new P-Q mixed modular equations. Using these theorems, we calculate several new explicit values of $a_{m, n}$ for $n = 3, 5, 7, 13$ and even values of $m$.
\end{abstract}


\maketitle
\section{Introduction}
The following definitions of theta functions \cite{Berndt-notebook-3} $\varphi$, $\psi$ and $f$ with $|q|<1$ are classical:
\begin{eqnarray*}
\varphi(q) &=& f(q,q)=\sum^{\infty}_{n=-\infty}q^{n^2}=(-q;q^2)^2_{\infty}(q^2;q^2)_{\infty},\\
\psi(q) &=& f(q,q^3)=\sum^{\infty}_{n=0}q^{\frac{n(n+1)}{2}}=\frac{(q^2;q^2)_{\infty}}{(q;q^2)^2_{\infty}},\\
f(-q)&=&f(-q,-q^2)=\sum^{\infty}_{n=-\infty}(-1)^nq^{\frac{n(3n-1)}{2}}=(q;q)_{\infty},
\end{eqnarray*}
where, $\displaystyle (a;q)_{\infty}=\prod^{\infty}_{n=0}\left(1-aq^n\right)$.\\
Now, we shall recall the definition of modular equation from \cite{Berndt-notebook-3}. The complete elliptic integral of the first kind $K(k)$ of modulus $k$ is defined by
\begin{align}\label{112235eq0}
K(k)=\int_{0}^{\frac{\pi}{2}}\frac{d\theta}{\sqrt{1-k^2\sin^2 \theta}}=\frac{\pi}{2}\sum^{\infty}_{n=0}\frac{\left(\frac{1}{2}\right)^2_n}{(n!)^2}k^{2n}=
\frac{\pi}{2}\varphi^2\left(e^{-\pi\frac{K'}{K}}\right),
\quad (0<k<1)
\end{align}
and let $K'=K(k'),$ where $k'=\sqrt{1-k^2}$ is called the complementary modulus of $k$. Let $K, K', L$ and $L'$ denote the complete elliptic integrals of the first kind associated with the moduli $k, k', l,$ and $l'$ respectively. Say, if the equality
\begin{equation} \label{eq00}
n\frac{K'}{K}=\frac{L'}{L}
\end{equation}
holds for some positive integer $n$, then a modular equation of degree $n$ will be the relation between the moduli $k$ and $l$ which is implied by equation (\ref{eq00}). Ramanujan defined his modular equation involving $\alpha$ and $\beta$, where $\alpha=k^2$ and $\beta=l^2$. Then, we considered $\beta$ as degree $n$ over $\alpha$.

For $q=e^{-\pi\sqrt{n}}$, Weber-Ramanujan class invariants \cite[p.183, (1.3)]{Berndt-notebook-5} are defined by
\begin{align}\label{eq1}
G_n = 2^{-1/4}q^{-1/24} \chi(q)
\quad ; \quad
g_n = 2^{-1/4}q^{-1/24} \chi(-q) ,
\end{align}
where, $n$ is a positive rational number and $\chi(q)=(-q;q^2)_\infty$. By Entry 24(iii)\cite[p.39]{Berndt-notebook-3}, $G_n$ and $g_n$ can be represented as
\begin{align}\label{eqgn1}
G_n = \frac{f(q)}{2^{1/4}q^{1/24}f(-q^2)}
\quad ; \quad
g_n = \frac{f(-q)}{2^{1/4}q^{1/24}f(-q^2)}.
\end{align}
On page 338, in his first notebook \cite{sr1}, Ramanujan defines
\begin{equation} \label{eq1}
 a_{m,n} =\frac{nq^{\frac{(n-1)}{4}}\psi^2\left(q^n\right)\varphi^2\left(-q^{2n}\right)}{\psi^2\left(q\right)\varphi^2\left(-q^2\right)},
\end{equation}
where, $\displaystyle q=e^{-\pi\sqrt{m/n}}$ and $m$, $n$ are positive rationals then, on page 338 and 339, he offered a list of 18 particular values. All those 18 values were established by Berndt, Chan and Zhang \cite{be5}. Recently, Prabhakaran and Ranjith Kumar \cite{djpkrk} have established a new general formulae for the explicit evaluations of $a_{3m,3}$ and $a_{m,9}$ by using $P-Q$ mixed modular equations and the values for certain class invariant of Ramanujan. Also they have calculated some new explicit values of $a_{3m,3}$ for $m = 2, 7, 13, 17, 25, 37$ and $a_{m,9}$ for $m = 17, 37$.\\
Naika and Dharmendra \cite{ms2} gave alternative form of (\ref{eq1}) as follows
\begin{equation} \label{eq2}
 a_{m,n}=\frac{nq^{\frac{(n-1)}{4}}\psi^2\left(-q^n\right)\varphi^2\left(q^{n}\right)}{\psi^2\left(-q\right)\varphi^2\left(q\right)}.
 \end{equation}
They proved the general theorems to calculate explicit values of $a_{m,n}$.  Closely associated with $a_{m,n},$ which is the parameter $b_{m,n}$ introduced by Naika et al \cite{3ms2} and it is defined as
\begin{eqnarray}\label{mu}
 b_{m,n}=\frac{nq^{\frac{(n-1)}{4}}\psi^2\left(q^n\right)\varphi^2\left(-q^{n}\right)}{\psi^2\left(q\right)\varphi^2\left(-q\right)}.
\end{eqnarray}
Also, they established general formulas for explicit evaluations of $b_{m,n}$ by using various degrees of modular equations, along with class invariant $g_n$. Furthermore, they computed particular values of $b_{m,n}$.

The organisation of the present study is as follows. In section 2, we collect some identities which will be used to prove the main results. Then in section 3, we establish several new $P-Q$ mixed modular equations. Applying these modular equations, we establish the general theorems for the explicit evaluations of the Ramanujan's remarkable product of theta functions along with class invariant $g_n$. Finally, explicit values of $a_{m, n}$ for $n= 3, 5, 7, 13$ and several even values of $m$ have been evaluated, which are presented in Section 4.

\section{Preliminaries}
In this section,  we have listed some lemmas which play a vital role in establishing the main results of the present study.
\begin{lemma}\cite[Entry 24(iii)  p.\ 39]{Berndt-notebook-3} We have
\begin{eqnarray} \label{1lem}
  f(q)f(-q^2)&=& \psi(-q)\varphi(q).
\end{eqnarray}
\end{lemma}
\begin{lemma}\cite[Entry 8(iii), (iv)  p.\ 376]{Berndt-notebook-3} We have
\begin{eqnarray} \label{p2beta}
m &=& \left(\frac{\beta}{\alpha}\right)^{1/4}+\left(\frac{1-\beta}{1-\alpha}\right)^{1/4}
-\left(\frac{\beta(1-\beta)}{\alpha(1-\alpha)}\right)^{1/4}-4\left(\frac{\beta(1-\beta)}{\alpha(1-\alpha)}\right)^{1/6},\\ \nonumber
\textrm{and}\\ \label{p2meta}
\frac{13}{m} &=& \left(\frac{\alpha}{\beta}\right)^{1/4}+\left(\frac{1-\alpha}{1-\beta}\right)^{1/4}
-\left(\frac{\alpha(1-\alpha)}{\beta(1-\beta)}\right)^{1/4}-4\left(\frac{\alpha(1-\alpha)}{\beta(1-\beta)}\right)^{1/6}.
\end{eqnarray}
\end{lemma}
\begin{lemma}\cite[Entries 51, 53, 55, 57    p. \ 204, 206, 209, 211]{Berndt-notebook-4} We have
\begin{enumerate}
\item [\rm{(i)}]  If  $P=\displaystyle \frac{f^2(-q)}{q^{1/6}f^2(-q^{3})}$  and
    $ Q=\displaystyle \frac{f^2(-q^2)}{q^{1/3}f^2(-q^{6})}$, then,
\begin{eqnarray}\label{lemma1}
PQ+\frac{9}{PQ} &=&\left(\frac{P}{Q}\right)^3+\left(\frac{Q}{P}\right)^3.
\end{eqnarray}
\item [\rm{(ii)}]  If  $P=\displaystyle \frac{f(-q)}{q^{1/6}f(-q^{5})}$  and
    $ Q=\displaystyle \frac{f(-q^2)}{q^{1/3}f(-q^{10})}$, then,
\begin{eqnarray}\label{lemma2}
PQ+\frac{5}{PQ} &=&\left(\frac{P}{Q}\right)^3+\left(\frac{Q}{P}\right)^3.
\end{eqnarray}
\item [\rm{(iii)}]  If  $P=\displaystyle \frac{f^2(-q)}{q^{1/2}f^2(-q^{7})}$  and
    $ Q=\displaystyle \frac{f^2(-q^2)}{qf^2(-q^{14})}$, then,
\begin{eqnarray}\label{lemma3}
PQ+\frac{49}{PQ} &=&\left(\frac{P}{Q}\right)^3+\left(\frac{Q}{P}\right)^3-8\left(\frac{P}{Q}+\frac{Q}{P}\right).
\end{eqnarray}
\item [\rm{(iv)}]  If  $P=\displaystyle \frac{f(-q)}{q^{1/2}f(-q^{13})}$  and
    $ Q=\displaystyle \frac{f(-q^2)}{qf(-q^{26})}$, then,
\begin{eqnarray}\label{lemma4}
PQ+\frac{13}{PQ} &=&\left(\frac{P}{Q}\right)^3+\left(\frac{Q}{P}\right)^3-4\left(\frac{P}{Q}+\frac{Q}{P}\right).
\end{eqnarray}
\end{enumerate}
\end{lemma}
\begin{lemma}\cite[Theorems 3.2.4, 3.2.6, 3.2.7, 3.5.2]{y1} We have
\begin{enumerate}
\item [\rm{(i)}]  If  $P=\displaystyle \frac{f(-q)}{q^{1/24}f(-q^{2})}$  and
    $ Q=\displaystyle \frac{f(-q^3)}{q^{1/8}f(-q^{6})}$, then,
\begin{eqnarray}\label{lemma5}
\left(PQ\right)^3+\left(\frac{2}{PQ}\right)^3 &=&\left(\frac{Q}{P}\right)^6-\left(\frac{P}{Q}\right)^6.
\end{eqnarray}
\item [\rm{(ii)}]  If  $P=\displaystyle \frac{f(-q)}{q^{1/24}f(-q^{2})}$  and
    $ Q=\displaystyle \frac{f(-q^5)}{q^{5/24}f(-q^{10})}$, then,
\begin{eqnarray}\label{lemma6}
\left(PQ\right)^2+\left(\frac{2}{PQ}\right)^2 &=&\left(\frac{Q}{P}\right)^3-\left(\frac{P}{Q}\right)^3.
\end{eqnarray}
\item [\rm{(iii)}]  If  $P=\displaystyle \frac{f(-q)}{q^{1/24}f(-q^{2})}$  and
    $ Q=\displaystyle \frac{f(-q^7)}{q^{7/24}f(-q^{14})}$, then,
\begin{eqnarray}\label{lemma7}
\left(PQ\right)^3+\left(\frac{2}{PQ}\right)^3 &=&\left(\frac{P}{Q}\right)^4+\left(\frac{Q}{P}\right)^4-7.
\end{eqnarray}
\item [\rm{(iv)}]  If  $P=\displaystyle \frac{f(-q)}{q^{1/8}f(-q^{4})}$  and
    $ Q=\displaystyle \frac{f(-q^3)}{q^{3/8}f(-q^{12})}$, then,
\begin{eqnarray}\label{lemma7.1}
PQ+\frac{4}{PQ} &=&\left(\frac{P}{Q}\right)^2+\left(\frac{Q}{P}\right)^2.
\end{eqnarray}
\end{enumerate}
\end{lemma}
\begin{lemma}\cite[Theorems 2.1, 2.2, 2.3(i)]{vasukirk} We have
\begin{enumerate}
\item [\rm{(i)}]  If  $P=\displaystyle \frac{f(-q)f(-q^{3})}{q^{\frac{1}{6}}f(-q^{2})f(-q^{6})}$ and $Q=\displaystyle \frac{f(-q^2)f(-q^{6})}{q^{\frac{1}{3}}f(-q^{4})f(-q^{12})}$, then,
\begin{eqnarray}\label{lemma8}
Q^4-P^4Q^2-4P^2&=&0.
\end{eqnarray}
\item [\rm{(ii)}]  If  $P=\displaystyle \frac{f(-q)f(-q^{5})}{q^{\frac{1}{4}}f(-q^{2})f(-q^{10})}$ and $Q=\displaystyle \frac{f(-q^2)f(-q^{10})}{q^{\frac{1}{2}}f(-q^{4})f(-q^{20})}$, then,
\begin{eqnarray}\label{lemma9}
Q^8-P^8Q^4-8(PQ)^4-16P^4&=&0.
\end{eqnarray}
\item [\rm{(iii)}]  If  $P=\displaystyle \frac{f(-q)f(-q^{7})}{q^{\frac{1}{3}}f(-q^{2})f(-q^{14})}$ and $Q=\displaystyle \frac{f(-q^2)f(-q^{14})}{q^{\frac{2}{3}}f(-q^{4})f(-q^{28})}$, then,
\begin{eqnarray}\label{lemma10}
Q^2-P^2Q-2P&=&0.
\end{eqnarray}
\end{enumerate}
\end{lemma}
\begin{lemma}\cite[(2.33) and (2.51)]{vasuki} We have
\begin{enumerate}
\item [\rm{(i)}]  If  $P=\displaystyle \frac{q^{1/6}f(-q)f(-q^{10})}{f(-q^{2})f(-q^{5})}$ and $ Q=\displaystyle \frac{q^{1/3}f(-q^2)f(-q^{20})}{f(-q^{4})f(-q^{10})}$, then,
\begin{eqnarray}\label{lemma11}
\left(\frac{1}{PQ}+PQ\right)^2&=&\left(\frac{1}{PQ}+PQ\right)\left(\left(\frac{P}{Q}\right)^3+\left(\frac{Q}{P}\right)^3\right)+4.
\end{eqnarray}
\item [\rm{(ii)}]  If  $P=\displaystyle \frac{q^{1/4}f(-q)f(-q^{14})}{f(-q^{2})f(-q^{7})}$ and $ Q=\displaystyle \frac{q^{1/2}f(-q^2)f(-q^{28})}{f(-q^{4})f(-q^{14})}$, then,
\begin{eqnarray}\nonumber
\left(\left(\frac{1}{PQ}\right)^4+\left(PQ\right)^4\right)+8\left(\left(\frac{P}{Q}\right)^4+\left(\frac{Q}{P}\right)^4\right)&=&
\left(\left(\frac{1}{PQ}\right)^2+\left(PQ\right)^2\right)\left(\left(\frac{P}{Q}\right)^6+\left(\frac{Q}{P}\right)^6\label{lemma12} \right.\\ &&\left. +
8\left(\frac{P}{Q}\right)^2+8\left(\frac{Q}{P}\right)^2\right)-18.
\end{eqnarray}
\end{enumerate}
\end{lemma}
\section{New $P-Q$ mixed modular equations }
In this section, we establish several new $P-Q$ mixed modular equations.
\begin{theorem}\label{tnbrde2}
  If  $P=\displaystyle \frac{q^{1/2}f(-q)f(-q^{26})}{f(-q^{2})f(-q^{13})}$  and  $ Q=\displaystyle \frac{qf(-q^{2})f(-q^{52})}{f(-q^{4})f(-q^{26})}$, then,
\begin{eqnarray}\label{eihma}
\left(PQ+\frac{1}{PQ}\right)^2
&=&\left(PQ+\frac{1}{PQ}\right)\left(\left(\frac{P}{Q}+\frac{Q}{P}\right)^3+\left(\frac{P}{Q}+\frac{Q}{P}\right)\right)
-4\left(\frac{P}{Q}+\frac{Q}{P}\right)^2+4.
\end{eqnarray}
\end{theorem}
\begin{proof}
Let
\begin{align*}\label{alp2beta}
R = \frac{q^{1/2}\chi(q)}{\chi(q^{13})}
\quad ; \quad
K= \frac{q^{1/2}\chi(-q^2)}{\chi(-q^{26})}.
\end{align*}
Transcribing $R$ and $Q$ using Entry 12(v) and (vii) of Chapter 17\cite[p.124]{Berndt-notebook-3} and simplifying, we arrive at
\begin{eqnarray}\label{alp2beta}
RK = \left(\frac{\beta}{\alpha}\right)^{1/8}
\quad \mbox{;}\quad
\frac{R^2}{K}= \left(\frac{1-\beta}{1-\alpha}\right)^{1/8}.
\end{eqnarray}
Employing \eqref{alp2beta} in \eqref{p2beta}, \eqref{p2meta}  and then eliminate $m$, we deduce that
\begin{eqnarray*}\nonumber
K^2R^8+K^6R^6-4K^4R^6+4K^2R^6-R^6+4K^6R^4-6K^4R^4+4K^2R^4-K^8R^2 &&\\+4K^6R^2-4K^4R^2+K^2R^2+K^6 &=& 0
\end{eqnarray*}
Dividing the aforementioned equation by $R^4K^4$ and rearranging the terms, then replacing $q$ by $-q$ and employing Entry 24(iii) of Chapter 16\cite[p.39]{Berndt-notebook-3}, the desired result was obtained.
\end{proof}
\begin{theorem}\label{tbghge2}
If  $P=\displaystyle \frac{f(-q)f(-q^{4})}{q^{5/2}f(-q^{13})f(-q^{52})}$ and $ Q=\displaystyle \frac{f(-q^2)f(-q^{2})}{q^{2}f(-q^{26})f(-q^{26})},$ then,
\begin{eqnarray}\nonumber
\left(PQ+\frac{169}{PQ}\right)^2&=&\left(PQ+\frac{169}{PQ}\right)
\left(\left(\frac{P}{Q}+\frac{Q}{P}\right)^9-\left(\frac{P}{Q}+\frac{Q}{P}\right)^7-21\left(\frac{P}{Q}+\frac{Q}{P}\right)^5
+35\left(\frac{P}{Q}+\frac{Q}{P}\right)^3
\right.\\ && \nonumber \left.+30\left(\frac{P}{Q}+\frac{Q}{P}\right)\right)
-\left(\frac{P}{Q}+\frac{Q}{P}\right)^{12}-4\left(\frac{P}{Q}+\frac{Q}{P}\right)^{10} +26\left(\frac{P}{Q}+\frac{Q}{P}\right)^8 -
89\left(\frac{P}{Q}+\frac{Q}{P}\right)^6
\\ && +829\left(\frac{P}{Q}+\frac{Q}{P}\right)^4-1821\left(\frac{P}{Q}+\frac{Q}{P}\right)^2+576. \label{eihma1}
\end{eqnarray}
\end{theorem}
\begin{proof}
Replacing  $q$ by $q^2$  in \eqref{lemma4}, we obtain that
\begin{eqnarray} \nonumber
\left(\frac{f(-q^2)f(-q^4)}{q^{3}f(-q^{26})f(-q^{52})}\right)+13\left(\frac{q^{3}f(-q^{26})f(-q^{52})}
{f(-q^2)f(-q^4)}\right)&=&\left(\frac{qf(-q^2)f(-q^{52})}{f(-q^4)f(-q^{26})}+\frac{f(-q^4)f(-q^{26})}{qf(-q^2)f(-q^{52})}\right)^3 \\&&
-7\left(\frac{qf(-q^2)f(-q^{52})}{f(-q^4)f(-q^{26})}+\frac{f(-q^4)f(-q^{26})}{qf(-q^2)f(-q^{52})}\right). \label{lemma 2.2}
\end{eqnarray}
Now multiplying \eqref{lemma4} with \eqref{lemma 2.2}, after further simplification, we deduce that
\begin{equation} \label{eq3.6}
PQ+\frac{169}{PQ} = x^3-4\left(\frac{P}{Q}+\frac{Q}{P}\right)x^2-4\left(\frac{P}{Q}+\frac{Q}{P}\right)^2x
+8x+\left(\frac{P}{Q}+\frac{Q}{P}\right)^3+21\left(\frac{P}{Q}+\frac{Q}{P}\right),
\end{equation}
where, $\displaystyle x= \frac{q^{\frac{3}{2}}f(-q)f(-q^{52})}{f(-q^4)f(-q^{13})}
+\frac{f(-q^4)f(-q^{13})}{q^{\frac{3}{2}}f(-q)f(-q^{52})}.$\\
Solving \eqref{eihma} for $x$ and choosing the appropriate root, we have
\begin{eqnarray}\label{lemma 62.2}
x &=& \frac{u^3+u+\sqrt{u^6+2u^4-15u^2-16}}{2},
\end{eqnarray}
where, $u = P/Q+Q/P$. Employing \eqref{lemma 62.2}  in \eqref{eq3.6}, after a straightforward, lengthly calculation, we arrive at desire result.
\end{proof}
\begin{theorem}
If  $P=\displaystyle \frac{f(-q)f(-q^{4})}{q^{5/4}f(-q^{7})f(-q^{28})}$ and $ Q=\displaystyle \frac{f(-q^2)f(-q^{2})}{qf(-q^{14})f(-q^{14})},$ then,
\begin{eqnarray}\nonumber
\left(\left(PQ\right)^2+\left(\frac{49}{PQ}\right)^2\right)^2&=&\left(\left(PQ\right)^2+\left(\frac{49}{PQ}\right)^2\right)
\left(\left(\left(\frac{P}{Q}\right)^2+\left(\frac{Q}{P}\right)^2\right)^9+7\left(\left(\frac{P}{Q}\right)^2+\left(\frac{Q}{P}\right)^2\right)^7
 \right.\\ &&\left. \nonumber -37\left(\left(\frac{P}{Q}\right)^2+\left(\frac{Q}{P}\right)^2\right)^5 -77\left(\left(\frac{P}{Q}\right)^2+\left(\frac{Q}{P}\right)^2\right)^3
+294\left(\left(\frac{P}{Q}\right)^2+\left(\frac{Q}{P}\right)^2\right)\right)
 \\ &&   \nonumber -\left(\left(\frac{P}{Q}\right)^2+\left(\frac{Q}{P}\right)^2\right)^{12} -20\left(\left(\frac{P}{Q}\right)^2+\left(\frac{Q}{P}\right)^2\right)^{10}-86\left(\left(\frac{P}{Q}\right)^2+\left(\frac{Q}{P}\right)^2\right)^8
\\  && +16317\left(\left(\frac{P}{Q}\right)^2+\left(\frac{Q}{P}\right)^2\right)^4 \nonumber
-2065\left(\left(\frac{P}{Q}\right)^2+\left(\frac{Q}{P}\right)^2\right)^6 \\ &&
-22981\left(\left(\frac{P}{Q}\right)^2+\left(\frac{Q}{P}\right)^2\right)^2.\label{7gm}
\end{eqnarray}
\end{theorem}
\begin{proof}
The proof our theorem can be obtained by \eqref{lemma3} and \eqref{lemma12}. Since the proof is analogous to the aforementioned theorem, and so we omit the details.
\end{proof}
\begin{theorem}
If $P=\displaystyle \frac{f(-q)f(-q^{4})}{q^{5/6}f(-q^{5})f(-q^{20})}$ and $ Q=\displaystyle \frac{f(-q^2)f(-q^{2})}{q^{2/3}f(-q^{10})f(-q^{10})},$ then,
\begin{eqnarray}\nonumber
\left(PQ+\frac{25}{PQ}\right)^2&=&\left(PQ+\frac{25}{PQ}\right)\left(\left(\left(\frac{P}{Q}\right)^3+\left(\frac{Q}{P}\right)^3\right)^3
+6\left(\left(\frac{P}{Q}\right)^3+\left(\frac{Q}{P}\right)^3\right)\right)-\left(\left(\frac{P}{Q}\right)^3+\left(\frac{Q}{P}\right)^3\right)^4\\&& \label{smsks}
-37\left(\left(\frac{P}{Q}\right)^3+\left(\frac{Q}{P}\right)^3\right)^2+64.
\end{eqnarray}
\end{theorem}
\begin{proof}
The proof our theorem can be obtained by \eqref{lemma2} and \eqref{lemma11}. Since the proof is analogous to Theorem \ref{tbghge2}, and so we omit the details.
\end{proof}
\begin{theorem}\label{namjfrt3e2}
  If  $P=\displaystyle \frac{f(-q)f(-q^{2})}{q^{1/4}f(-q^{3})f(-q^{6})}$  and  $ Q=\displaystyle \frac{f(-q^2)f(-q^{4})}{q^{1/2}f(-q^{6})f(-q^{12})}$, then,
\begin{eqnarray}\nonumber
\left(\left(PQ\right)^2+\left(\frac{9}{PQ}\right)^2\right)\left(\left(\frac{P}{Q}\right)^2+\left(\frac{Q}{P}\right)^2\right)
&=&\left(\left(\frac{P}{Q}\right)^2+\left(\frac{Q}{P}\right)^2\right)^4 \\ && -11\left(\left(\frac{P}{Q}\right)^2+\left(\frac{Q}{P}\right)^2\right)^2-8. \label{eillhma}
\end{eqnarray}
\end{theorem}
\begin{proof}
The following equations can be obtained from \eqref{lemma1}, we have
\begin{eqnarray}
P^2+\frac{9}{P^2} &=& \left(\frac{q^{1/12}f(-q)f(-q^{6})}{f(-q^{2})f(-q^{3})}\right)^6
+\left(\frac{f(-q^{2})f(-q^{3})}{q^{1/12}f(-q)f(-q^{6})}\right)^6 \label{dpak}\\
Q^2+\frac{9}{Q^2} &=& \left(\frac{q^{1/6}f(-q^2)f(-q^{12})}{f(-q^{4})f(-q^{6})}\right)^6
+\left(\frac{f(-q^{4})f(-q^{6})}{q^{1/6}f(-q^2)f(-q^{12})}\right)^6 \label{dpaqk}
\end{eqnarray}
Now, multiplying \eqref{dpak} with \eqref{dpaqk}, we deduce that
\begin{eqnarray}\label{y68}
y^6+\frac{1}{y^6}&=& \left(PQ\right)^2+\left(\frac{9}{PQ}\right)^2+9\left(\frac{P}{Q}\right)^2+9\left(\frac{Q}{P}\right)^2
-\left(\frac{P}{Q}\right)^6-\left(\frac{Q}{P}\right)^6,
\end{eqnarray}
where, $\displaystyle y = \frac{f(-q)f(-q^4)f(-q^{6})f(-q^{6})}{q^{1/12}f(-q^{2})f(-q^{2})f(-q^{3})f(-q^{12})}$. \\ Similarly, the following equation can be obtained by \eqref{lemma5} and \eqref{lemma8}, we obtain that
\begin{eqnarray}\label{yhbg}
y^6+\frac{1}{y^6}&=& \left(\frac{P}{Q}\right)^6+\left(\frac{Q}{P}\right)^6-z^3+4z-8z^{-1}
\end{eqnarray}
where, $\displaystyle z = \left(\frac{f(-q^2)f(-q^2)f(-q^{6})f(-q^{6})}{q^{1/6}f(-q)f(-q^{4})f(-q^{3})f(-q^{12})}\right)^3$.\\
Employing \eqref{lemma7.1} in \eqref{lemma8}, we deduce that
\begin{eqnarray}\label{nxmhf}
z&=& \left(\frac{P}{Q}\right)^2+\left(\frac{Q}{P}\right)^2.
\end{eqnarray}
When combining \eqref{y68}, \eqref{yhbg}, and \eqref{nxmhf}, we arrive at the desired result.
\end{proof}
\begin{remark}
The alternative proof of the above theorem can be found in \cite{ms7}.
\end{remark}
\begin{theorem} \label{lema13th}
  If  $P=\displaystyle \frac{f(-q)}{q^{1/24}f(-q^{2})}$ and $ Q=\displaystyle \frac{f(-q^{13})}{q^{13/24}f(-q^{26})},$ then,
\begin{eqnarray*}\label{lema13}
\left(PQ\right)^6+\left(\frac{2}{PQ}\right)^6&=& \left(\frac{Q}{P}\right)^7-\left(\frac{P}{Q}\right)^7
-13\left(\left(\frac{Q}{P}\right)^5-\left(\frac{P}{Q}\right)^5\right)
+52\left(\left(\frac{Q}{P}\right)^3-\left(\frac{P}{Q}\right)^3\right)
-78\left(\frac{Q}{P}-\frac{P}{Q}\right).
\end{eqnarray*}
\end{theorem}
\begin{proof}
Transcribing Entry 41(ii) of Chapter 36 \cite[p.378]{Berndt-notebook-5} by employing Entries 12(i),(iii) of Chapter 17 \cite[p.124]{Berndt-notebook-3}, and then changing $q$ to $-q$, we can obtain the desired result.
\end{proof}
\section{ Explicit evaluations of  $a_{m, n}$.}
In this section, we establish new general theorems for the explicit evaluations of Ramanujan's remarkable product of theta functions $a_{m, n}$ for $n = 3, 5, 7, 13, $ and even values of $m$. On applying these theorems, we compute several explicit evaluations of $a_{m, n}$.
\begin{theorem}\label{temjhgf2}
If $m$ is any positive rational, and
\begin{eqnarray*}
\Lambda &=& \left(\sqrt{2}g_{3m}g_{m/3}\right)^3\label{3k.0m},
\end{eqnarray*}
then,
\begin{eqnarray}
\frac{a_{m, 3}}{b_{4m, 3}}-\frac{b_{4m, 3}}{a_{m, 3}}&=& \Lambda \label{3m},\\
\frac{1}{a_{m, 3}b_{4m, 3}}-a_{m, 3}b_{4m, 3} &=&\frac{1}{9}\left(\frac{\Lambda^4+11\Lambda^2-8}{\Lambda}\right).\label{3.0m}
\end{eqnarray}
\end{theorem}
\begin{proof}
Applying \eqref{1lem} in \eqref{eq2} and \eqref{mu}, we deduce that
\begin{align}\label{a55lta}
\sqrt{\frac{n}{a_{m, n}}} =\frac{f(q)f(-q^{2})}{q^{(n-1)/8}f(q^{n})f(-q^{2n})}
\quad ; \quad
\sqrt{\frac{n}{b_{m, n}}} =\frac{f(-q)f(-q^{2})}{q^{(n-1)/8}f(-q^{n})f(-q^{2n})}.
\end{align}
Combining Entry 24(iii) and 24(iv)\cite[p.39]{Berndt-notebook-3}, we obtain that
\begin{eqnarray}\label{entry24}
\frac{f(q)}{f(-q^2)}&=& \frac{f(-q^2)f(-q^2)}{f(-q)f(-q^4)}.
\end{eqnarray}
Now, employing \eqref{entry24} in \eqref{nxmhf} and replaceing $q$ by $-q$ after that applying \eqref{eqgn1} and \eqref{a55lta} with $n = 3 $ and $q=e^{-\pi\sqrt{m/3}}$, we arrive at \eqref{3m}. Similarly, from \eqref{eillhma}, we arrive at \eqref{3.0m}.
\end{proof}
\begin{corollary}\label{coro248} We have
\begin{eqnarray*}
a_{2, 3} &=& \left(\sqrt{2}-1\right)\left(\sqrt{3}+\sqrt{2}\right)^{1/2},\\
a_{4, 3} &=& \left(\sqrt{2}+1\right)\left(\frac{\sqrt{3}-1}{\sqrt{2}}\right)^{5/2},\\
a_{8, 3} &=& \left(\sqrt{2}-1\right)^{3/2}\left(\sqrt{3}-\sqrt{2}\right)^{5/4}\left(\sqrt{\sqrt{3}+2}+\sqrt{\sqrt{3}+1}\right).
\end{eqnarray*}
\end{corollary}
\begin{proof}
Letting $m=2$, $g_6g_{2/3}=1$ by Corollary 2.2.4 \cite{y1} and  employing the value in \eqref{3m} and \eqref{3.0m}, then solving quadratic equations  and choosing the appropriate root, we obtain $a_{2, 3}$. Now, applying \eqref{eqgn1} in \eqref{lemma8}, we obtain that
\begin{eqnarray}\label{6th11113}
\left(g_{4n}g_{36n}\right)^4-2\left(g_{4n}g_{36n}\right)^2\left(g_{n}g_{9n}\right)^4-2\left(g_{n}g_{9n}\right)^2 &=& 0.
\end{eqnarray}
Applying $n=1/3$ in \eqref{6th11113} and employing the relation $g_3g_{1/3}=\left(g_{12}g_{4/3}\right)^{-1}$, it can be obtained by Corollary 2.2.4 \cite{y1}, then, we deduce the following
\begin{align}\label{alta}
g_{12}g_{4/3} = 2^{1/3},
\quad \mbox{and also}\quad
\Lambda =4\sqrt{2}.
\end{align}
By applying \eqref{3m} and \eqref{3.0m} with $m =4$, then, we have $a_{4, 3}$. The explicit values of $g_{24}g_{8/3}$ can be evaluated  by \eqref{6th11113} with $n = 2/3$ .  As the proof of $a_{8, 3}$ being similar to the proof of the $a_{4, 3}$. So we omit the details.
\end{proof}
\begin{corollary}\label{cjor} We have
\begin{eqnarray*}
a_{10, 3} &=& \left(\sqrt{2}-1\right)^2\left(\sqrt{6}+\sqrt{5}\right)^{1/2}\left(\frac{\sqrt{3}+1}{\sqrt{2}}\right)
\left(\frac{\sqrt{5}-1}{2}\right)^3\label{103}.
\end{eqnarray*}
\end{corollary}
\begin{proof}
Employing \eqref{eqgn1} in \eqref{lemma5} with $q=e^{-\pi \sqrt{n/3}}$, we deduce that
\begin{eqnarray}\label{g3n}
2\sqrt{2}\left(\left(g_{3n}g_{n/3}\right)^3+\left(g_{3n}g_{n/3}\right)^{-3}\right)
&=& \left( \frac{g_{3n}}{g_{n/3}}\right)^6 - \left( \frac{g_{n/3}}{g_{3n}}\right)^6.
\end{eqnarray}
From the table in Chapter 34 of Ramanujan’s notebooks \cite[p.200]{Berndt-notebook-5}, we have
\begin{eqnarray}
g_{30} &=& \left(\sqrt{5}+2\right)^{1/6}\left(\sqrt{10}+3\right)^{1/6}.\label{1thec003}
\end{eqnarray}
Using \eqref{1thec003} in \eqref{g3n} with $n = 10$, we verify that
\begin{eqnarray}
g_{10/3} &=& \left(\sqrt{5}-2\right)^{1/6}\left(\sqrt{10}+3\right)^{1/6}.\label{2thec003}
\end{eqnarray}
Employing \eqref{1thec003} and \eqref{2thec003} in \eqref{3m}, then solving quadratic equation, and choosing the appropriate root, we deduce that
\begin{eqnarray}
\frac{a_{10, 3}}{b_{40, 3}}&=& 2\sqrt{5}+3\sqrt{2}+\sqrt{39+12\sqrt{10}}.\label{2theco003}
\end{eqnarray}
By $(9.5)$ \cite[ p.284]{Berndt-notebook-3}, we have
\begin{eqnarray}
\sqrt{39+12\sqrt{10}} &=& \sqrt{15}+2\sqrt{6}.\label{2themmco003}
\end{eqnarray}
From \eqref{3.0m}, the following equation is obtained
\begin{eqnarray}
\frac{1}{a_{10, 3}b_{40, 3}}-a_{10, 3}b_{40, 3}&=& \frac{136 \sqrt{10}+432}{\sqrt{2}}.\label{2t03}
\end{eqnarray}
Solving \eqref{2t03} and choosing the convenient root, we deduce that
\begin{eqnarray}
a_{10, 3}b_{40, 3}&=& \left(\sqrt{2}-1\right)^4\left(\frac{\sqrt{5}-1}{2}\right)^6.\label{k003}
\end{eqnarray}
Combining \eqref{2theco003}, \eqref{2themmco003}, and \eqref{k003}, we arrive at the desired result.
\end{proof}
\begin{corollary} We have
\begin{eqnarray*}
a_{6, 3} &=& \frac{1}{3}\left(\sqrt{3}-\sqrt{2}\right)\left(\sqrt{2}+1\right)^{1/2}\left(\frac{\sqrt{3}+1}{\sqrt{2}}\right),\\
a_{14, 3} &=& \left(\sqrt{2}+1\right)\left(\sqrt{3}-\sqrt{2}\right)^2\left(\sqrt{8}-\sqrt{7}\right)\left(\sqrt{7}+\sqrt{6}\right)^{1/2},\label{143}\\
a_{26, 3} &=& \left(\sqrt{2}-1\right)^4\left(\sqrt{3}+\sqrt{2}\right)\left(\sqrt{26}-5\right)\left(3\sqrt{3}+\sqrt{26}\right)^{1/2},\label{263}\\
a_{34, 3} &=&\left(\sqrt{2}-1\right)^3\left(\sqrt{17}-4\right)^2\left(\sqrt{51}+5\sqrt{2}\right)^{1/2}\left(\frac{\sqrt{3}+1}{\sqrt{2}}\right)^2. \label{343}
\end{eqnarray*}
\end{corollary}
\begin{proof}
Employing class invariant for $m = 2, 18, 42, 78, $ and $ 102$ \cite[p.200-202]{Berndt-notebook-5} in Theorem \ref{temjhgf2}, we obtain all the those values mentioned. Since the proof is analogous to the previous corollary,  and so we omit the details.
\end{proof}
\begin{remark}
The different proof of $a_{2, 3}$ and $a_{6, 3}$ can be found in \cite{ms7} and \cite{djpkrk} respectively.
\end{remark}
\begin{theorem}\label{t5mce2}
If $m$ is any positive rational, and
\begin{eqnarray*}
\Lambda &=& \left(\frac{g_{5m}}{g_{m/5}}\right)^3,
\end{eqnarray*}
then,
\begin{eqnarray*}
\left(\sqrt{\frac{a_{m, 5}}{b_{4m, 5}}}-\sqrt{\frac{b_{4m, 5}}{a_{m, 5}}}\right)^2&=&
\left(\sqrt{\frac{a_{m, 5}}{b_{4m, 5}}}-\sqrt{\frac{b_{4m, 5}}{a_{m, 5}}}\right)\left(\Lambda-\frac{1}{\Lambda}\right)-4,\\
25\left(\frac{1}{\sqrt{a_{m, 5}b_{4m, 5}}}-\sqrt{a_{m, 5}b_{4m, 5}}\right)^2&=& 5\left(\frac{1}{\sqrt{a_{m, 5}b_{4m, 5}}}-\sqrt{a_{m, 5}b_{4m, 5}}\right)\left(\left(\Lambda-\frac{1}{\Lambda}\right)^3
-6\left(\Lambda-\frac{1}{\Lambda}\right)\right) \\ && +\left(\Lambda-\frac{1}{\Lambda}\right)^4-37\left(\Lambda-\frac{1}{\Lambda}\right)^2-64.
\end{eqnarray*}
\end{theorem}
\begin{proof}
The proof of theorem can be obtained by \eqref{lemma11} and \eqref{smsks} with $n =5 $ and $q=e^{-\pi\sqrt{m/5}}$. Since the proof is analogous to Theorem \ref{temjhgf2}, and so, we omit the details.
\end{proof}
\begin{corollary} We have
\begin{eqnarray*}
a_{2, 5} &=& \left(\sqrt{2}+1\right)\left(\frac{\sqrt{5}-1}{2}\right)^3,\\
a_{4, 5} &=& \left(\sqrt{2}+1\right)^{1/2}\left(\sqrt{10}+3\right)^{1/2}\left(\sqrt{\frac{45\sqrt{5}+103}{4}}-\sqrt{\frac{45\sqrt{5}+99}{4}}\right),\\
a_{6, 5} &=&\left(\sqrt{2}-1\right)^2\left(\sqrt{10}-3\right)\left(\frac{\sqrt{3}+1}{\sqrt{2}}\right)
\left(\frac{\sqrt{5}+\sqrt{3}}{\sqrt{2}}\right),\\
a_{8, 5} &=& \left(\sqrt{2}-1\right)^{5/2}\left(\frac{\sqrt{5}-1}{2}\right)^{9/2}\left(\sqrt{2\sqrt{10}+7}+\sqrt{2\sqrt{10}+6}\right),\\
a_{14, 5} &=&\left(\sqrt{2}-1\right)^3\left(\sqrt{10}-3\right)^2\left(\frac{3+\sqrt{7}}{\sqrt{2}}\right)
\left(\frac{\sqrt{7}+\sqrt{5}}{\sqrt{2}}\right),\\
a_{26, 5} &=&\left(\sqrt{2}+1\right)^2\left(\sqrt{10}+3\right)\left(\sqrt{65}-8\right)^2\left(\frac{\sqrt{13}-3}{2}\right)^2,\\
a_{38, 5} &=&\left(\sqrt{2}-1\right)^8\left(2\sqrt{5}+19\right)\left(\sqrt{19}+3\sqrt{2}\right)\left(\frac{\sqrt{5}-1}{2}\right)^9.
\end{eqnarray*}
\end{corollary}
\begin{proof}
The aforementioned values can be obtained by Theorem \ref{t5mce2}, \eqref{lemma9}, \eqref{lemma6}, and the explicit values of $g_n$ for $n = 30, 70, 130,$ and $ 190$ \cite[p.200-203]{Berndt-notebook-5}. Since the proof is analogous to the  Corollary \ref{coro248}, and Corollary \ref{cjor}  and so, we omit the details.
\end{proof}
\begin{remark}
The alternative proof of $a_{2, 5}$ can be found in \cite{ms8}.
\end{remark}
\begin{theorem}\label{te7m2}
If $m$ is any positive rational, and
\begin{eqnarray*}
\Lambda &=& \left(\frac{g_{7m}}{g_{m/7}}\right)^2,
\end{eqnarray*}
then,
\begin{eqnarray*}
\left(\frac{a_{m, 7}}{b_{4m, 7}}-\frac{b_{4m, 7}}{a_{m, 7}}\right)^2&=&
\left(\frac{a_{m, 7}}{b_{4m, 7}}-\frac{b_{4m, 7}}{a_{m, 7}}\right)\left(\left(\Lambda-\frac{1}{\Lambda}\right)^3-5\left(\Lambda-\frac{1}{\Lambda}\right)\right)-8\left(\Lambda-\frac{1}{\Lambda}\right)^2,\\
2401\left(\frac{1}{a_{m, 7}b_{4m, 7}}-a_{m, 7}b_{4m, 7}\right)^2&=&49\left(\frac{1}{a_{m, 7}b_{4m, 7}}-a_{m, 7}b_{4m, 7}\right)\left(\left(\Lambda-\frac{1}{\Lambda}\right)^9-7\left(\Lambda-\frac{1}{\Lambda}\right)^7-37\left(\Lambda-\frac{1}{\Lambda}\right)^5
\right.\\ &&\left.+77\left(\Lambda-\frac{1}{\Lambda}\right)^3+294\left(\Lambda-\frac{1}{\Lambda}\right)\right)+\left(\Lambda-\frac{1}{\Lambda}\right)^{12}
-20\left(\Lambda-\frac{1}{\Lambda}\right)^{10}\\ &&+86\left(\Lambda-\frac{1}{\Lambda}\right)^8-2065\left(\Lambda-\frac{1}{\Lambda}\right)^6
-16317\left(\Lambda-\frac{1}{\Lambda}\right)^4-22981\left(\Lambda-\frac{1}{\Lambda}\right)^2.
\end{eqnarray*}
\end{theorem}
\begin{proof}
The proof of theorem can be obtained by \eqref{lemma12}, and \eqref{7gm} with $n =7 $ and $q=e^{-\pi\sqrt{m/7}}$. Since the proof is analogous to Theorem \ref{temjhgf2}, and so, we omit the details.
\end{proof}
\begin{corollary} We have
\begin{eqnarray*}
a_{2, 7} &=& \left(\sqrt{8}+\sqrt{7}\right)^{1/4}\left(\frac{3+\sqrt{7}}{\sqrt{2}}\right)^{1/2}
\left(\sqrt{\frac{\sqrt{2}+3}{2}}-\sqrt{\frac{\sqrt{2}+1}{2}}\right)^2,\\
a_{4, 7} &=& \left(\sqrt{8}+\sqrt{7}\right)\left(\frac{3-\sqrt{7}}{\sqrt{2}}\right)^{5/2}, \\
a_{6, 7} &=& \left(\sqrt{3}-\sqrt{2}\right)^2\left(2\sqrt{7}-3\sqrt{3}\right)\left(\sqrt{7}+\sqrt{6}\right)^{1/2}
\left(\frac{\sqrt{3}+1}{\sqrt{2}}\right)^2,\\
a_{10, 7} &=& \left(\sqrt{10}-3\right)^2\left(3\sqrt{14}+5\sqrt{5}\right)^{1/2}\left(\frac{\sqrt{7}+\sqrt{5}}{\sqrt{2}}\right)
\left(\frac{\sqrt{5}-1}{2}\right)^6.
\end{eqnarray*}
\end{corollary}
\begin{proof}
The aforementioned values can be obtained by Theorem \ref{te7m2}, \eqref{lemma10}, \eqref{lemma7}, Lemma 9.11 \cite[p.292]{Berndt-notebook-5}, and the explicit values of $g_n$ for $n = 42, $ and $ 70$ \cite[p.201]{Berndt-notebook-5}. Since the proof is analogous to the  Corollary \ref{coro248}, and Corollary \ref{cjor} and so, we omit the details.
\end{proof}
\begin{theorem}\label{te2}
If $m$ is any positive rational, and
\begin{eqnarray*}\label{123}
\Lambda &=& \frac{g_{13m}}{g_{m/13}},
\end{eqnarray*}
then,
\begin{eqnarray}\nonumber
\left(\sqrt{\frac{a_{m, 13}}{b_{4m, 13}}}-\sqrt{\frac{b_{4m, 13}}{a_{m, 13}}}\right)^2&=&
\left(\sqrt{\frac{a_{m, 13}}{b_{4m, 13}}}-\sqrt{\frac{b_{4m, 13}}{a_{m, 13}}}\right)
\left(\left(\Lambda-\frac{1}{\Lambda}\right)^3-\left(\Lambda-\frac{1}{\Lambda}\right)\right)\\ && -4\left(\Lambda-\frac{1}{\Lambda}\right)^2-4, \label{13m}\\ \nonumber
169\left(\frac{1}{\sqrt{a_{m, 13}b_{4m, 13}}}-\sqrt{a_{m, 13}b_{4m, 13}}\right)^2&=&
13\left(\frac{1}{\sqrt{a_{m, 13}b_{4m, 13}}}-\sqrt{a_{m, 13}b_{4m, 13}}\right)
\left(\left(\Lambda-\frac{1}{\Lambda}\right)^9\right.\\ \nonumber &&\left.+\left(\Lambda-\frac{1}{\Lambda}\right)^7-21\left(\Lambda-\frac{1}{\Lambda}\right)^5
-35\left(\Lambda-\frac{1}{\Lambda}\right)^3+30\left(\Lambda-\frac{1}{\Lambda}\right)\right)
\\ &&  +\left(\Lambda-\frac{1}{\Lambda}\right)^{12}-4\left(\Lambda-\frac{1}{\Lambda}\right)^{10}-26\left(\Lambda-\frac{1}{\Lambda}\right)^8 -89\left(\Lambda-\frac{1}{\Lambda}\right)^6  \nonumber\\ && -829\left(\Lambda-\frac{1}{\Lambda}\right)^4
-1821\left(\Lambda-\frac{1}{\Lambda}\right)^2-576. \label{13m1}
\end{eqnarray}
\end{theorem}
\begin{proof}
The proof of the theorem can be obtained by \eqref{eihma} and \eqref{eihma1} with $n =13 $ and $q=e^{-\pi\sqrt{m/13}}$. Since the proof is analogous to Theorem \ref{temjhgf2}, and so, we omit the details.
\end{proof}
\begin{corollary} We have
\begin{eqnarray*}
a_{6, 13} &=& \left(\sqrt{3}+\sqrt{2}\right)\left(\sqrt{13}+2\sqrt{3}\right)\left(\sqrt{2}-1\right)^4
\left(\frac{\sqrt{13}-3}{2}\right)^3.
\end{eqnarray*}
\end{corollary}
\begin{proof}
Employing \eqref{eqgn1} in Theorem \ref{lema13th}, we can deduce, after simplification,
\begin{eqnarray}\nonumber
8\left(\left(g_{13n}g_{n/13}\right)^6+\left(g_{13n}g_{n/13}\right)^{-6}\right) &=& \left(\frac{g_{13n}}{g_{n/13}}-\frac{g_{n/13}}{g_{13n}}\right)^7-6\left(\frac{g_{13n}}{g_{n/13}}-\frac{g_{n/13}}{g_{13n}}\right)^5
+\left(\frac{g_{13n}}{g_{n/13}}-\frac{g_{n/13}}{g_{13n}}\right)^3\\ && \label{13} +20\left(\frac{g_{13n}}{g_{n/13}}-\frac{g_{n/13}}{g_{13n}}\right)
\end{eqnarray}
From the table in Chapter 34 of Ramanujan’s notebooks \cite[p.202]{Berndt-notebook-5},
\begin{eqnarray}
g_{78} &=& \left(\sqrt{26}+5\right)^{1/6}\left(\frac{\sqrt{13}+3}{2}\right)^{1/2}.\label{1tooec003}
\end{eqnarray}
Using \eqref{1tooec003} in \eqref{13} with $n = 6$, we find that
\begin{eqnarray}
g_{6/13} &=&  \left(\sqrt{26}+5\right)^{1/6}\left(\frac{\sqrt{13}-3}{2}\right)^{1/2}.\label{ec003}
\end{eqnarray}
Employing \eqref{1tooec003} and \eqref{ec003} in \eqref{13m}, with $m=6$, then solving quadratic equation and choosing the appropriate root, we deduce that
\begin{eqnarray}
\sqrt{\frac{a_{6, 13}}{b_{24, 13}}}&=&\left(\sqrt{3}+\sqrt{2}\right)\left(\sqrt{13}+2\sqrt{3}\right).\label{213theco003}
\end{eqnarray}
By \eqref{13m1}, we obtain that
\begin{eqnarray}
\sqrt{a_{6, 13}b_{24, 13}}&=& \left(\sqrt{2}-1\right)^4
\left(\frac{\sqrt{13}-3}{2}\right)^3.\label{k113003}
\end{eqnarray}
Combining \eqref{213theco003}, and \eqref{k113003}, we arrive at the desired result.
\end{proof}
\begin{corollary} We have
\begin{eqnarray*}
a_{10, 13} &=&\left(\sqrt{26}+5\right)\left(\sqrt{2}+1\right)^2\left(\sqrt{65}-8\right)^2\left(\frac{\sqrt{5}-1}{2}\right)^9
\end{eqnarray*}
\end{corollary}
\begin{proof}
Employing the explicit value of $g_{130}$ \cite[p.203]{Berndt-notebook-5} in Theorem \ref{te2} with $m=10$, we obtain desired result. Since the proof is analogous to the previous corollary,  and so we omit the details.
\end{proof}
We conclude the present work with some remarks. The mixed modular equations in theorems \ref{tnbrde2} - \ref{namjfrt3e2}, we  can be obtain general theorems for the explicit evaluations of $b_{m, n}$ and $b_{4m, n}$ involving class invariant $G_n$. Ramanujan recorded enormous explicit values of $G_n$ for odd values of $n$ and $g_n$ for even values of $n$ \cite[p. 189-204]{Berndt-notebook-5}. We utilize the values of $G_n$ to calculate several explicit values of $b_{m, n}$, for $m$ and $n$ are odd. On application of these values, we can calculate several explicit values of $g_n$ for odd value of $n$ by using the some identities. These identities can be found in \cite{3ms2}. Finally, we conclude that only a few works \cite{ms9,y1} computed $g_n$ for odd value of $n$.


\begin{thebibliography}{00}

\bibitem{Berndt-notebook-3}  Berndt, B.C.: {\it Ramanujan's notebooks. Part III}, Springer-Verlag, New York, 1991.

\bibitem{Berndt-notebook-4} Berndt, B.C.: {\it Ramanujan's notebooks. Part IV}, Springer-Verlag, New York, 1994.

\bibitem{Berndt-notebook-5} Berndt, B.C.: {\it Ramanujan's notebooks. Part V}, Springer-Verlag, New York, 1998.

\bibitem{be5} Berndt, B.C., Chan, H.H., Zhang, L.-C.:  Ramanujan's remarkable product of theta-functions, Proc. Edinburgh Math. Soc. (2) {\bf 40} (1997), no.~3, 583--612. doi:10.1017/S0013091500024032

\bibitem{ms9} Mahadeva Naika, M.S., Bairy, K.S.:  On some new explicit evaluations of class invariants, Vietnam J. Math. {\bf 36} (2008), no.~1, 103--124.

\bibitem{ms7} Mahadeva Naika, M.S., Bairy, K.S., Suman, N.P.:  Certain modular relations for remarkable product of theta-functions, Proc. Jangjeon Math. Soc. {\bf 17} (2014), no.~3, 317--331.

\bibitem{ms8} Mahadeva Naika, M.S., Chandankumar, S., Harish, M.: On some new $P$-$Q$ mixed modular equations, Ann. Univ. Ferrara Sez. VII Sci. Mat. {\bf 59} (2013), no.~2, 353--374. doi:10.1007/s11565-013-0183-y

\bibitem{ms2} Mahadeva Naika, M.S., Dharmendra, B.N.:  On some new general theorems for the explicit evaluations of Ramanujan's remarkable product of theta-functions, Ramanujan J. {\bf 15} (2008), no.~3, 349--366. doi:10.1007/s11139-007-9081-1

\bibitem{3ms2} Mahadeva Naika, M.S., Maheshkumar, M.C., Bairy, K.S.: On some remarkable product of theta-function, Aust. J. Math. Anal. Appl. {\bf 5} (2008), no.~1, Art. 13, 1--15.

\bibitem{djpkrk} Prabhakaran, D.J., Ranjithkumar, K.: The explicit formulae and evauations of Ramanujan's remarkable product of theta-functions, In press. doi:10.11650/tjm/190706

\bibitem{sr1} Ramanujan, S.: Notebooks (2 volumes), Tata Institudte of Fundamental Research, Bombay 1957.

\bibitem{vasukirk} Vasuki, K.R., Shivashankara, K.: A note on explicit evaluations of products and ratios of class invariants, Math. Forum {\bf 13} (1999/00), 45--56.

\bibitem{vasuki} Vasuki, K.R., Sreeramamurthy, T. G.:  Certain new Ramanujan's Schl\"{a}fli-type mixed modular equations, J. Math. Anal. Appl. {\bf 309} (2005), no.~1, 238--255. doi:10.1016/j.jmaa.2005.01.035

\bibitem{y1} Yi, J. : Construction and application of modular equation Ph.D thesis, University of Illionis 2001.


\end{thebibliography}
\end{document}